\documentclass{article}
\usepackage[all]{xy}

\usepackage{latexsym,amssymb,chicago,geometry,graphicx}
\usepackage{xspace}
\usepackage{amsthm}
\usepackage{amsfonts}
\usepackage{amsmath}
\usepackage{mathrsfs}
\mathcode`\:="603A     % : = \colon
\mathcode`\<="4268     % < = \langle
\mathcode`\>="5269     % > = \rangle
\mathchardef\gt="313E  % arithmetic
\mathchardef\lt="313C  % strict order
\mathchardef\colon="303A  % :=
\edef\cdrestoreat{%%
\noexpand\catcode\lq\noexpand\@=\the\catcode\lq\@}\catcode\lq\@=11
\def\thmitem{\def\theenumi{\@roman\c@enumi}
\def\labelenumi{{\normalfont(\theenumi)}}}
\expandafter\ifx\csname bezier\endcsname\relax
\newcounter{P@sc}\newcounter{P@scp}\newcounter{P@t}\newlength{\P@x}
\newlength{\P@xa}\newlength{\P@xb}\newlength{\P@y}\newlength{\P@ya}
\newlength{\P@yb}\newsavebox{\P@pt}
\def\bezier#1(#2,#3)(#4,#5)(#6,#7){\c@P@sc#1\relax
 \c@P@scp\c@P@sc \advance\c@P@scp\@ne
 \P@xb #4\unitlength \advance\P@xb -#2\unitlength \multiply\P@xb \tw@
 \P@xa #6\unitlength \advance\P@xa -#2\unitlength
 \advance\P@xa -\P@xb \divide\P@xa\c@P@sc
 \P@yb #5\unitlength \advance\P@yb -#3\unitlength \multiply\P@yb \tw@
 \P@ya #7\unitlength \advance\P@ya -#3\unitlength
 \advance\P@ya -\P@yb \divide\P@ya\c@P@sc
 \setbox\P@pt\hbox{\vrule height\@halfwidth depth\@halfwidth 
 width\@wholewidth}\c@P@t\z@ 
 \put(#2,#3){\@whilenum{\c@P@t<\c@P@scp}\do
 {\P@x\c@P@t\P@xa \advance\P@x\P@xb \divide\P@x\c@P@sc \multiply\P@x\c@P@t 
 \P@y\c@P@t\P@ya \advance\P@y\P@yb \divide\P@y\c@P@sc \multiply\P@y\c@P@t 
 \raise \P@y \hbox to \z@{\hskip \P@x\unhcopy\P@pt\hss}\advance\c@P@t\@ne}}}
\else\relax\fi
\def\sig{\mbox{\vbox to 0pt{\vss\hbox to 0pt{\setlength{\unitlength}{.35cm}
\begin{picture}(5,5)
      \bezier140(1,3)(1,0)(1.5,0) % top-down P
      \bezier120(1.5,0)(1.5,1.5)(.5,2.7) % bottom-up P
      \bezier180(.5,2.7)(0,4)(4,3.5) % over top P
      \bezier140(4,3.5)(4.5,3.2)(1.5,2.7) % rounding P
      \multiput(2,.5)(.4,.05)3{% doing i, n
      \bezier30(0,0)(.2,0)(.2,.3) % up leg
      \bezier30(.2,.3)(.2,.05)(.4,.05) % down leg
      }
      \bezier30(3.2,.65)(3.27,.67)(3.3,.8) % up short leg
      \bezier30(3.3,.8)(3.34,.6)(3.53,.66) % bottom o
      \bezier30(3.53,.66)(3.6,.73)(3.58,.8) % right o
      \bezier30(3.58,.8)(3.555,.99)(3.4,.96) % top o
      \bezier10(2.3,1)(2.3,1.01)(2.31,1.02) % dotting i
\end{picture}\hss}}}}
\let\ttori\tt
\def\tt{\@ifnextchar_{\top\kern-.5ex}{\top}}
\cdrestoreat

\newtheoremstyle{teorema}{\topsep}{\topsep}
{\thmitem\slshape}{}{\bf}{{\normalfont.}}{.5em}{}
\newtheoremstyle{definizione}{\topsep}{\topsep}
{\thmitem\normalfont}{}{\bf}{{\normalfont.}}{.5em}{}

\def\namefont#1{{\sc #1}}
\swapnumbers
\theoremstyle{teorema}
\newtheorem{theorem}{\namefont{Theorem}}[section]

\newtheorem{prop}[theorem]{\namefont{Proposition}}
\newtheorem{cor}[theorem]{\namefont{Corollary}}
\theoremstyle{definizione}
\newtheorem{definition}[theorem]{\namefont{Definition}}
\newtheorem{remark}[theorem]{\namefont{Remark}}

\newtheorem{exm}[theorem]{\namefont{Example}}
\newtheorem{exms}[theorem]{\namefont{Examples}}

\def\dfn#1{{\bfseries\itshape #1\/}}

\def\cat#1{\ensuremath{\mathcal{#1}}}
\def\Cat#1{\ensuremath{{\normalfont\textsf{\bfseries #1}}}}
\def\id#1{\ensuremath{\mathrm{id}_{#1}}}

\def\op{{}^{\textrm{\scriptsize op}}}

\def\Op{^{\textrm{\tiny op}}}
\def\sop{\strut^{\textrm{\scriptsize op}}}
\def\op{\mathchoice{\sop}{\sop}{\Op}{\Op}}

\def\exr{_{\textrm{\scriptsize ex/reg}}}

\def\Gr(#1){\ensuremath{\mathcal{G}_{#1}}}
\def\des#1{\ensuremath{\mathcal{D}\kern-.3ex\textit{es\kern.2ex}_{#1}}}

\let\Land\wedge

\let\ForalL\forall \def\Forall#1.{\ForalL_{#1}}
\let\ExistS\exists \def\Exists#1.{\ExistS_{#1}}
\def\LGE{\mathrel{\bgroup\ooalign{\hfil\raise.8ex\hbox{$\lt$}\hfil
 \crcr\hfil\raise-.5ex\hbox{$\eqslantgtr$}}\egroup}}
\def\Lge{\mathrel{\bgroup
 \ooalign{\hfil\raise.4ex\hbox{$\scriptscriptstyle\lt$}\hfil
 \crcr\hfil\raise-.3ex\hbox{$\scriptscriptstyle\eqslantgtr$}}\egroup}}

\def\gel{=}

\def\pr{\mathrm{pr}}
\def\cmp#1{\ensuremath{\{\kern-2.5pt|{#1}|\kern-2.5pt\}}}
\def\EH{\Cat{ED}\xspace}

\def\FS{\Cat{LFS}\xspace}
\def\RE{\Cat{Reg}\xspace}
\def\XC{\Cat{Xct}\xspace}
\def\CM{\Cat{CEED}\xspace}
\def\EE{\Cat{EED}\xspace}

\def\Qex#1#2{\relax\ifx#2q\ensuremath{\left[{#1}\right]}\fi
\ifx#2c\ensuremath{{\left|{#1}\right|}}\fi
\ifx#2x\ensuremath{\widehat{#1}}\fi}
\def\QEx#1#2{\ensuremath{{({#1})_{_{\textrm{\scriptsize #2}}}}}}

\def\Q#1{\Qex{#1}{q}}
\def\X#1{\QEX{#1}{x}}
\def\P#1{\Qex{#1}{c}}
\def\Q#1{\QEx{#1}{q}}
\def\Q#1{\ensuremath{\widehat{#1}}}
\def\X#1{\QEx{#1}{x}}
\def\P#1{\QEx{#1}{c}}

\def\D{{\rotatebox[origin=c]{180}{\ensuremath{E}}\kern-.3ex}}
\def\B{{\rotatebox[origin=c]{180}{\ensuremath{A}}\kern-.6ex}}
\def\QD{\Q{{\rotatebox[origin=c]{180}{\ensuremath{E}}}}\kern-.3ex}

\def\EC#1{\ensuremath
 {\left\lfloor\left.\kern-.2ex{#1}\kern-.2ex\right\rceil\right.}}
\def\Strut{\hbox{\vrule height.75em depth.35em width0pt}}
\def\ec#1{\ensuremath
 {\left\lfloor\left.\kern-.3ex\Strut{#1}\Strut\kern-.3ex\right\rceil\right.}}
\def\ec#1{\ensuremath{\left[{#1}\right]}}

\def\dpd{\mathchoice{\textstyle\prod}{\prod}{\prod}{\prod}\kern-.2ex\strut}
\def\F{F^{ML}}

\def\mtt{\ensuremath{\mathbf{mtt}}\xspace}
\def\CTT{\textrm{CoC}\xspace}
\def\prp{\kern.5ex\mbox{\normalfont\it prop}\kern.5ex}
\def\prps{\kern.5ex\mbox{\normalfont\it prop}_s\kern.5ex}

\def\blank{\mbox{--}}
\def\ie{{\textit{i.e.}}\xspace}
\def\eg{{\textit{e.g.}}\xspace}
\def\loccit{{\textit{loc.cit.}}\xspace}

\long\def\beginskip#1\endskip{\newpage}
\def\TP#1{\ensuremath{\cat{T}_{#1}}}
\def\ER#1{\ensuremath{\cat{E}_{#1}}}
\def\ple#1{\ensuremath{(#1)}}
\def\fp#1{\ensuremath{P_{#1}}}
\def\rla#1{\ar[#1]|{\smash{{}_{\scriptscriptstyle|}}}}
\def\Sub{\mathrm{Sub}}
\def\S#1{\ensuremath{\Sub_{\cat{#1}}}}
\def\twoup#1#2{\mathbin{\begin{array}[b]{@{}l@{}}
\kern.2ex\scriptstyle#1\\[-2.1ex]#2\end{array}}}
\def\tr#1{\ensuremath{{#1}^t}}

\CompileMatrices
\date{}
\begin{document}

\title{Unifying exact completions}
\author{Maria Emilia Maietti\thanks{%
Dipartimento di Matematica Pura ed Applicata,
Universit\`a degli Studi di Padova,
via Trieste 63, 35121 Padova, Italy,
email:~\texttt{maietti@math.unipd.it}}
\and
Giuseppe Rosolini\thanks{%
Dipartimento di Matematica, Universit\`a degli Studi di Genova,
via Dodecaneso 35, 16146 Genova, Italy,
email:~\texttt{rosolini@unige.it}}}
\maketitle

\begin{abstract}
We define the notion of exact completion with respect to an
existential elementary doctrine. We observe that 
the forgetful functor from the 2-category exact categories to
existential elementary doctrines has a left biadjoint that can be
obtained as a composite of two others. Finally, we conclude 
how this notion encompasses both that of the exact completion of a
regular category as well as that of the exact completion of a
cartesian category with weak pullbacks.
\end{abstract}



\section{Introduction}
The notion of completion by quotients, and in particular that of exact
completion, has been widely studied in category theory, see for
example \cite{JacobsB:catltt,CarboniA:freecl,CarboniA:regec}. The
concept of quotient completion is pervasive  not only
in mathematics but also in computer science, in particular
for what concerns the way proofs are formalized in a computer-assisted
way in an intensional set theory that does not carry quotient sets as
primitive notion.

In \cite{MaiettiME:quofcm} the authors began to study a
categorical structure involved with quotient completions, 
relativizing the basic concept to a doctrine equipped with a logical
structure sufficient to describe the notion of an equivalence relation.
The notion of quotient within an elementary doctrine
and that of elementary quotient completion producing a quotient completion
that is not generally exact but encompasses relevant examples used
in type theory were introduced in \cite{MaiettiME:eleqc}.

In the present paper, that analysis of quotient completion
is pushed further viewing the exact completion of a regular category
or the exact completions of a cartesian category with weak pullbacks
as instances of a more general ``exact completion''
with respect to an elementary existential doctrine.
 
Indeed, for an exact category \cat{X}, the indexed inf-semilattice
$\S{X}:\cat{X}\op\longrightarrow\Cat{InfSL}$ 
of subobjects, which assigns to an
object $A$ in \cat{X} the poset $S(A)$ of subobjects of $A$ in
\cat{X}, 
constitutes the archetypal example of a fibrations of sets and
functions as all known frameworks for modelling a constructive
theory of sets produce exact categories, \eg toposes as models of IZF
or arising from a tripos, 
categories of classes for CZF, total setoids \`a la Bishop on
Martin-{\"o}f's type theory~\cite{PMTT}. Since, within a set theory,
functions are defined from the logic, it is of
little surprise that the models are obtained from 
indexed inf-semilattices which are existential elementary doctrines.

We show that many of the models are obtained as a free
construction. Indeed, 
the forgetful functor from the 2-category of exact categories to that
of existential elementary doctrines has a left biadjoint that can be
obtained as a composite of two others: the first adds (full)
comprehensions to an existential elementary doctrine, the other
turns an existential elementary doctrine with full comprehension into
(the fibration of subobjects of) an exact category, universally so.
In particular, when the second is applied to the doctrine of
subobjects of a regular category, it gives rise to its exact
completion, see \cite{FreydP:cata}.

For an existential elementary doctrine $P$, the elementary quotient
completion of $P$ presented in \cite{MaiettiME:eleqc} appears as a
subcategory of the exact completion of $P$ by the universal properties
of the various constructions involved. There are interesting cases
when that inclusion is an equivalence; for instance,
when $P$ is the poset indexed doctrine $\Psi_{\cal C}$ of weak
subobjects of a cartesian category $\cal C$ with weak pullbacks. Thus
also the exact completion on a cartesian category with weak pullbacks
is an instance of the exact 
completion of an elementary existential doctrine
as the elementary quotient completion of  $\Psi_{\cal C}$ coincides
with the exact completion of $\cal C$ as a weakly lex category, see
\loccit.

\section{Elementary existential doctrines}
A doctrine subsumes the basic categorical concept of a logic.
The notion was introduced, in a series of seminal papers, by
F.W.~Lawvere to synthetize the structural properties of logical
systems, see 
\cite{LawvereF:adjif,LawvereF:diaacc,LawvereF:equhcs}, see also
\cite{LawvereF:setfm,JacobsB:catltt} for a unified survey. Lawvere's
crucial intuition was to consider logical languages and theories as
fibrations to study their 2-categorical properties, \eg connectives
and quantifiers are determined by structural adjunctions.
%% It provides
%% the underlying mathematics to what comes in the literature 
%% erroneously
%% under the name of ``Curry-Howard isomorphism'', since it was not
%% introduced by either and it is not an isomorphism---at most, after
%% carefully determining the structure, one may spot some quotients.

Recall that
an \dfn{elementary doctrine} is an indexed inf-semilattice 
$P:\cat{C}\op\longrightarrow\Cat{InfSL}$
from (the opposite of) a category \cat{C} with binary products
to the category of inf-semilattices and homomorphisms
such that, for every object $A$ in \cat{C}, there is an
object $\delta_A$ in $P(A\times A)$ and\thmitem
\begin{enumerate}
\item the assignment
$$\D_{<\id{A},\id{A}>}(\alpha)\colon=
P_{\pr_1}(\alpha)\Land\delta_A$$
for $\alpha$ in $P(A)$ determines a left adjoint to 
$P_{<\id{A},\id{A}>}:P(A\times A)\to P(A)$---the action of a doctrine
$P$ on an arrow is written as $P_f$
\item for every map 
$e\colon=<\pr_1,\pr_2,\pr_2>:X\times A\to X\times A\times A$ in \cat{C},
the assignment
$$\D_{e}(\alpha)\colon=
P_{<\pr_1,\pr_2>}(\alpha)\Land_{A\times A}P_{<\pr_2,\pr_3>}(\delta_A)$$
for $\alpha$ in $P(X\times A)$ determines a left adjoint to 
$P_{e}:P(X\times A\times A)\to P(X\times A)$.
\end{enumerate}
Also recall from \loccit that 
an \dfn{existential doctrine} is an indexed inf-semilattice 
$P:\cat{C}\op\longrightarrow\Cat{InfSL}$
such that, for $A_1$ and $A_2$ in \cat{C} and projections
$\pr:A_1\times A_2\to A_i$, $i=1,2$,
the functors $P_{\pr_i}:P(A_i)\to P(A_1\times A_2)$ have a left adjoint
$\D_{\pr_i}$ which satisfy
\begin{description}
\item[\dfn{Beck-Chevalley condition}:] for any pullback diagram
$$\xymatrix{X'\ar[r]^{\pr'}\ar[d]_{f'}&A'\ar[d]^f\\X\ar[r]^{\pr}&A}$$
with $\pr$ a projection (hence also $\pr'$ a projection), for any
$\beta$ in $P(X)$, the canonical arrow 
$\D_{\pr'}P_{f'}(\beta)\leq P_f\D_\pr(\beta)$ in $P(A')$ is iso; 
\item[\dfn{Frobenius reciprocity}:] for $\pr:X\to A$ a projection,
$\alpha$ in $P(A)$, $\beta$ in $P(X)$, the canonical arrow
$\D_\pr(P_\pr(\alpha)\Land_A\beta)\leq\alpha\Land_X\D_\pr(\beta)$ in $P(A)$ is
iso.
\end{description}

\begin{remark}\label{rembic}
Note for an elementary doctrine
$P:\cat{C}\op\longrightarrow\Cat{InfSL}$ that, in case \cat{C}
has a terminal object, conditions (ii) entails condition (i).

Also, given $\alpha_1$ in $P(X_1\times Y_1)$ and $\alpha_2$ in 
$P(X_2\times Y_2)$, if one writes
$\alpha_1\boxtimes\alpha_2$ for the object  
$$P_{<\pr_1,\pr_3>}(\alpha_1)\Land P_{<\pr_2,\pr_4>}(\alpha_2)$$
in 
$P(X_1\times X_2\times Y_1\times Y_2)$ 
where $\pr_i, i=1,2,3,4$, are the projections from 
$X_1\times X_2\times Y_1\times Y_2$ to each of the four factors,
then
condition (ii) is to require that
$\delta_{A\times B}=\delta_A\boxtimes\delta_B$ for every pair of
objects $A$ and $B$ in \cat{C}.
\end{remark}

Beyond the standard example of the elementary existential doctrine of
subobjects of a regular category \cat{X}, one can consider
examples directly from logic such as the indexed
Lindenbaum-Tarski algebras $LT:\cat{V}\op\longrightarrow\Cat{InfSL}$
of well-formed formulae of a theory
$\mathscr{T}$ with equality in a first order language
$\mathscr{L}$ where
the domain category \cat{V} has lists of variables as objects and
term substitutions as arrows, with composition given by simultaneous
substitution; 
the functor $LT:\cat{V}\op\longrightarrow\Cat{InfSL}$ takes a list of
variables to the Lindenbaum-Tarski algebra of
equivalence classes of well-formed formulae of
$\mathscr{L}$ whose free variables are within $x_1$,\ldots,$x_n$.

An important example for theories developed for formalizing
constructive mathematics is the following: 
Consider a cartesian category \cat{S} with 
weak pullbacks and the functor of {\it weak subobjects}
$\Psi:\cat{S}\op\longrightarrow\Cat{InfSL}$
which evaluates, at an object $A$ of \cat{S}, as the poset reflection
of each comma category $\cat{S}/A$. The left adjoints are
computed by post-composition. We refer the reader to
\cite{MaiettiME:quofcm,MaiettiME:eleqc} for further details.

We consider the 2-category \EH has elementary doctrines as objects,
1-arrows are pairs $(F,b)$ 
$$
\xymatrix@C=4em@R=1em{
{\cat{C}\op}\ar[rd]^(.4){P}_(.4){}="P"\ar[dd]_{F\op}&\\
           & {\Cat{InfSL}}\\
{\cat{D}\op}\ar[ru]_(.4){R}^(.4){}="R"&\ar"P";"R"_b^{\kern-.4ex\cdot}}
$$
where the functor $F$ preserves products and, for every object $A$ in
\cat{C}, the functor $b_A:P(A)\to R(F(A))$ preserves all the
structure. More explicitly, $b_A$ preserves finite meets and, for
every object $A$ in \cat{C}, 
$b_{A\times A}(\delta_A)\gel R_{<F(\pr_1),F(\pr_2)>}(\delta_{F(A)})$
and the 2-arrows are natural transformations
$\theta:F\twoup{\kern.5ex\cdot}{\to}G$ such that
$$
\xymatrix@C=6em@R=1em{
{\cat{C}\op}\ar[rd]^(.4){P}_(.4){}="P"
\ar@<-1ex>@/_/[dd]_(.35){F\op}="F"\ar@<1ex>@/^/[dd]^(.35){G\op}="G"&\\
           & {\Cat{InfSL}}\\
{\cat{D}\op}\ar[ru]_(.4){R}^(.4){}="R"&
\ar@/_/"P";"R"_{b\kern.5ex\cdot\kern-.5ex}="b"
\ar@<1ex>@/^/"P";"R"^{\kern-.5ex\cdot\kern.5ex c}="c"
\ar"G";"F"_{.}^{\theta\op}\ar@{}"b";"c"|{\leq}}
$$
so that, for every object $A$ in \cat{C} and every $\alpha$ in $P(A)$,
one has $b_A(\alpha)\leq R_{\theta_A}(c_A(\alpha))$.

The 2-category \EE is the 1-full subcategory of \EH on elementary
existential doctrines where 2-arrows have each component $b_A$
preserving the existential adjoints.

As mentioned in the Introduction, since the indexed inf-semilattice 
$\S{X}:\cat{X}\op\longrightarrow\Cat{InfSL}$
of subobjects for an exact category \cat{X} is elementary 
existential, that construction induces an obvious forgetful functor from
the 2-category \XC of exact categories and regular functors to \EE.

In \cite{MaiettiME:eleqc} the authors presented a construction to
add quotients to an elementary doctrine freely. A similar construction
is that used to produce a topos from a tripos, see
\cite{HylandJ:trit,PittsA:triir,OostenJ:reaait}, and it produces a
left biadjoint to the forgetful functor from \XC to \EE.

\begin{definition}\label{tris}
Given an elementary existential doctrine
$P:\cat{C}\op\longrightarrow\Cat{InfSL}$, consider the category \TP{P}, called
{\it exact completion of the e.e.d. $P$},
whose
\begin{description}
\item[objects] are pairs $\ple{A,\rho}$ such that
$\rho$ is in $P(A\times A)$ and satisfies
$$\begin{array}{lp{25em}}
\rho\leq\fp{<p_2,p_1>}(\rho)&
in $P(A\times A)$ with $p_1,p_2:A\times A\to A$ projections\\[1ex]
\fp{<p_1,p_2>}(\rho)\Land\fp{<p_2,p_3>}(\rho)
\leq\fp{<p_1,p_3>}(\rho)&
in $P(A\times A\times A)$ with $p_1,p_2,p_3:A\times A\to A$
projections\\
\end{array}$$
\item[an arrow
$\phi:\ple{A,\rho}\to\ple{B,\sigma}$] is an object $\phi$ in
$P(A\times B)$ such that 
\begin{enumerate}
\item $\phi\leq\fp{<p_1,p_1>}(\rho)\Land\fp{<p_2,p_2>}(\sigma)$
\item $\fp{<p_1,p_2>}(\rho)\wedge \fp{<p_2,p_3>}(\phi)\leq
\fp{<p_1,p_3>}(\phi)$ in $P(A\times A\times B)$ where the $p_i$'s are
appropriate projections
\item $\fp{<p_1,p_2>}(\phi)\wedge\fp{<p_2,p_3>}(\sigma)\leq
\fp{<p_1,p_3>}(\phi)$ in $P(A\times B\times B)$ where, again, the
$p_i$'s are appropriate projections
\item $\fp{<p_1,p_2>}(\phi)\wedge\fp{<p_1,p_3>}(\phi)\leq
\fp{<p_2,p_3>}(\sigma)$ in $P(A\times B\times B)$ where the $p_i$'s are
as before
\item $\fp{<p_1,p_1>}(\rho)\leq\D_{p_2}(\phi)$ 
in $P(A)$ where $p_1:A\times B\to A$ and $p_2:A\times B\to B$ are the
projections
\end{enumerate}
\end{description}
where composition
$\xymatrix@1@=4ex{{\ple{A,\rho}}\ar[r]^{\phi}&{\ple{B,\sigma}}
\ar[r]^{\psi}&{\ple{C,\tau}}}$ is defined as
$$
\D_{p_2}(\fp{<p_1,p_2>}(\phi)\Land\fp{<p_2,p_3>}(\psi))
$$
and identity is
$\xymatrix@1@=4ex{{\ple{A,\rho}}\ar[r]^{\rho}&{\ple{A,\rho}}}$
\end{definition}

\begin{exms}
The main examples of this construction are toposes obtained from a
tripos, see \cite{HylandJ:trit,PittsA:triir,OostenJ:reaait}.
\end{exms}

\begin{remark}
It is quite apparent that the elementary structure plays no role in
the definitions in \ref{tris}---but it will be crucial for
\ref{mainthm}. We refer the reader to \cite{PasqualiF:cofced} for an
analysis of that.
\end{remark}

\begin{remark}
The logical relevance of \ref{tris} is exposed if
one considers the allegory $\cat{A}_P$ of relations of
an elementary existential doctrine
$P:\cat{C}^{\mathrm{op}}\longrightarrow\Cat{InfSL}$, see \cite{FreydP:cata},
whose objects are those of \cat{C} and the poset of 1-arrows from
$A$ to $B$ is $P(A\times B)$.
Composition of 1-arrows 
$\xymatrix@1@=4ex{A\ar[r]^{\theta}\rla{r}
&B\ar[r]^{\zeta}\rla{r}
&C}$ is
$$\D_{p_2}(\fp{<p_1,p_2>}(\theta)\Land\fp{<p_2,p_3>}(\zeta))$$
with identities given by $\delta_A$.
The opposite $\theta^\circ$ of a 1-arrow
$\xymatrix@1@=4ex{A\ar[r]^{\theta}\rla{r}&B}$
is given by $\xymatrix@1@=4ex{B\ar[rr]^{\fp{<p_2,p_1>}(\theta)}\rla{rr}&&A}$.

If one then takes maps in the splitting (allegory) of the ``symmetric
idempotents'' of  $\cat{A}_P$, one gets exactly the category \TP{P},
see \cite{CarboniA:catarp}.

The locally posetal category $\cat{A}_P$ is also a cartesian
bicategory, see \cite{CarboniA:carbi}.
The product functor of the base extends to a symmetric tensor
$\boxtimes$ as in \ref{rembic}. The
structure of commutative comonoid on each object $A$ is given by
$$\xymatrix@=2.5em{
1&A\ar[l]_{\tt}\rla{l}
\ar[rrr]^(.45){\fp{<p_1,p_2>}(\delta)\Land\fp{<p_1,p_3>}(\delta)}
\rla{rrr}&&&A\times A
}$$

Note that the computation of the opposite $\theta^\circ$ of a 1-arrow 
$\xymatrix@1@=4ex{A\ar[r]^{\theta}\rla{r}&B}$
in the cartesian bicategory gives precisely the 1-arrow
$\xymatrix@1@=4ex{B\ar[r]^{\fp{<p_2,p_1>}(\theta)}\rla{r}&A}$, see
\cite{CarboniA:carbi}.

The conditions (ii)-(v) in \ref{tris} are written in the notation of
the bicategory respectively as
$$\xymatrix@R=.1ex@C=3em{&A\ar@/^/[rd]^{\phi}|{\smash{{}_{\scriptscriptstyle|}}}\\
A\ar@{{}{}{}}[rr]|{\rotatebox[origin=c]{270}{$\leq$}}
\ar@/^/[ru]^{\rho}|{\smash{{}_{\scriptscriptstyle|}}}
\ar@/_1em/[rr]_{\phi}|{\smash{{}_{\scriptscriptstyle|}}}
&&B\\}\quad
\xymatrix@R=.1ex@C=3em{&B\ar@/^/[rd]^{\sigma}|{\smash{{}_{\scriptscriptstyle|}}}\\
A\ar@{{}{}{}}[rr]|{\rotatebox[origin=c]{270}{$\leq$}}
\ar@/^/[ru]^{\phi}|{\smash{{}_{\scriptscriptstyle|}}}
\ar@/_1em/[rr]_{\phi}|{\smash{{}_{\scriptscriptstyle|}}}
&&B}\quad
\xymatrix@R=.1ex@C=3em{&A\ar@/^/[rd]^{\phi}|{\smash{{}_{\scriptscriptstyle|}}}\\
B\ar@{{}{}{}}[rr]|{\rotatebox[origin=c]{270}{$\leq$}}
\ar@/^/[ru]^{\phi^\circ}|{\smash{{}_{\scriptscriptstyle|}}}
\ar@/_1em/[rr]_{\sigma}|{\smash{{}_{\scriptscriptstyle|}}}
&&B}\quad
\xymatrix@R=.1ex@C=3em{{\qquad}\\
A\ar@{{}{}{}}[rr]|{\rotatebox[origin=c]{270}{$\leq$}}
\ar@/_/[rd]_{\phi}|{\smash{{}_{\scriptscriptstyle|}}}
\ar@/^1em/[rr]^{\rho}|{\smash{{}_{\scriptscriptstyle|}}}
&&A\\
&B\ar@/_/[ru]_{\phi^\circ}|{\smash{{}_{\scriptscriptstyle|}}}}
$$
\end{remark}

We shall find it easy to obtain the construction of \TP{P} as the
composite of two left biadjoints to forgetful functors:
\begin{enumerate}
\item the left biadjoint to the inclusion of the 1-full
2-subcategory \CM of \EE on those elementary existential doctrines
with full comprehensions;
\item the left biadjoint to the forgetful functor from \XC to \CM
which takes an exact category \cat{X} to the doctrine
$\S{X}:\cat{X}\op\longrightarrow\Cat{InfSL}$ of subobjects of \cat{X}.
\end{enumerate}

\section{The left biadjoints}

Recall that, for a doctrine
$P:\cat{C}\op\longrightarrow\Cat{InfSL}$ and for an
object $\alpha$ in some $P(A)$, a \dfn{comprehensions} of $\alpha$ is
a map $\cmp\alpha:X\to A$ in \cat{C} such 
that $P_{\cmp\alpha}(\alpha)=\tt_X$ and, for every $f:Z\to A$ such
that $P_f(\alpha)=\tt_Z$ there is a unique map $g:Z\to X$ such that
$f=\cmp\alpha\circ g$. One says that $P$ \dfn{has comprehensions} if
every $\alpha$ has a comprehension, and that $P$ 
\dfn{has full comprehensions} if, moreover, $\alpha\leq\beta$ in
$P(A)$ whenever $\cmp\alpha$ factors through $\cmp\beta$.

As we may need also the weakened form of comprehension, recall that a
\dfn{weak comprehension} of $\alpha$ is a map $c:W\to A$ in \cat{C}
such  that $P_{c}(\alpha)=\tt_W$ and, for every $f:Z\to A$ such
that $P_f(\alpha)=\tt_Z$ there is a (not necessarily unique) map
$g:Z\to X$ such that $f=c\circ g$.

Recall from \cite{JacobsB:catltt}
that the fibration of vertical maps on the category of points
freely adds comprehensions to a given fibration producing an indexed
poset in case the given fibration is such.
For a doctrine 
$P:\cat{C}\op\longrightarrow\Cat{InfSL}$, the indexed poset 
consists of the base category of points \Gr(P) where
\begin{description}
\item[an object] is a pair $(A,\alpha)$ where
$A$ is in \cat{C} and $\alpha$ is in $P(A)$
\item[an arrow {$f:(A,\alpha)\to(B,\beta)$}] is an arrow $f:A\to B$ in
\cat{C} such that $\alpha\leq P_f(\beta)$.
\end{description}
Since the fibres of $P$ are inf-semilattices, 
the category \Gr(P) has products and there is a natural embedding 
$I:\cat{C}\to\Gr(P)$ which maps $A$ to $(A,\top_A)$. The indexed
functor extends to
$\P{P}:\Gr(P)\op\longrightarrow\Cat{InfSL}$
along $I$ by setting 
$\P{P}(A,\alpha)\colon=\{\gamma\in P(A)\mid
\gamma\leq\alpha\}$. Moreover, the comprehensions in \P{P} are full. 

\begin{theorem}\label{cthn}
There is a left bi-adjoint to the inclusion of \CM into \EE.
\end{theorem}

\begin{proof}
It is enough to check that, when
$P:\cat{C}\op\longrightarrow\Cat{InfSL}$ is existential,
the doctrine $\P{P}:\Gr(P)\op\longrightarrow\Cat{InfSL}$ is
existential and the pair $(I,\id{P}):P\to\P{P}$ preserves them.
\end{proof}

For the next step it is useful to recall three results about
fibrations with full comprehensions, regular and exact categories:

The first is in 
\cite{HughesJ:facsft}: in the notation introduced above, it states
that there is a biequivalence between \CM and the 2-category \FS of
categories with finite limits and a proper stable factorization system
(with left exact functors preserving the factorization). 

The second is
in \cite{KellyG:notrrf} and shows that the inclusion of the 2-category
\RE of regular categories (with regular functors) into \FS has a left
adjoint: the left biadjoint to the inclusion is computed on a category
\cat{B} with stable proper factorization system
$(\mathcal{E},\mathcal{M})$ as the
category of maps for the cartesian bicategory of
$\mathcal{M}$-relations in \cat{B}. 

The third is the result from
\cite{FreydP:cata} that the inclusion into \RE of the full
2-subcategory \XC on exact categories has a left biadjoint, which we
shall denote as $(\blank)\exr:\RE\longrightarrow\XC$.

The computation of the composite of the three left biadjoint produces
a 2-functor $\CM\longrightarrow\XC$ which, given an elementary
existential 
doctrine $P:\cat{C}\op\longrightarrow\Cat{InfSL}$, produces the full
subcategory \ER{P} of \TP{P} on those objects $\ple{A,\rho}$ such that
$$\tt_A\leq\fp{<\id{A},\id{A}>}\rho$$
---or, equivalently, $\delta_A\leq\rho$.

Following \cite{MaiettiME:quofcm} we shall
refer to such an object $\rho$ in $P(A\times A)$ as a
\dfn{$P$-equivalence relation on $A$}.
Condition \ref{tris}(i) for arrows in \ER{P} becomes redundant
and condition \ref{tris}(v) can be reduced to
$\tt_A\leq\D_{p_2}(\phi)$. 
For each object $A$ in \cat{C}, one can consider the object
$(A,\delta_A)$ in \ER{P}, and such assignment extends to a functor
$D:\cat{C}\to\ER{P}$ mapping an arrow $f:A\to B$ to the relation
$\D_{<\id{A},f>}(\tt_A)=\fp{(f\times\id{B})}(\delta_B)$. In turn, it
gives rise to a 1-arrow from $P$ to 
the indexed inf-semilattice of subobjects
$\Sub_{\ER{P}}:\ER{P}\op\longrightarrow\Cat{InfSL}$ since
$\Sub_{\ER{P}}(A,\delta_A)
\twoup{\sim}{\to}P(A)$.
%% thanks to the following lemma.

%% \begin{lemma}
%% If the elementary existential doctrine
%% $P:\cat{C}\op\longrightarrow\Cat{InfSL}$ has full comprehensions, then
%% every subobject of $(A,\delta_A)$ in \ER{P} has a representative which
%% is a comprehension from an object of the form $(X,\delta_X)$.
%% \end{lemma}

\begin{exms}
The leading example of the above construction  $\ER{P}$ is the exact completion $\cat{X}\exr$~\cite{FreydP:cata,CarboniA:somfcr,CarboniA:regec}   of a regular category
$\cat{X}$, which coincides with $\ER{\S{X}}$ for the doctrine $\S{X}:\cat{X}\op\longrightarrow\Cat{InfSL}$ of
subobjects of  \cat{X}. 
 
Other examples  come from 
theories apt to formalize constructive mathematics:
the category of total setoids \`a la Bishop and functional relations
 based on
the minimalist type theory in \cite{m09}, which coincides with the construction 
\ER{G^{\mtt}} where  the doctrine $G^{\mtt}$
is defined as in \cite{MaiettiME:quofcm}, or the category of total setoids \`a la Bishop and functional relations
based on the Calculus of Constructions~\cite{tc90}, which forms a topos as mentioned  in \cite{BCP:2003} and coincides with
 \ER{G^{\CTT}} 
 where the doctrine $G^{\CTT}$ is
constructed from the Calculus of Construction
as
$G^{\mtt}$.
\end{exms}

\begin{theorem}\label{mainthm}
For every elementary existential doctrine 
$P:\cat{C}\op\longrightarrow\Cat{InfSL}$
with full comprehensions, pre-composition with the 1-arrow
$$
\xymatrix@C=4em@R=1em{
{\cat{C}\op}\ar[rd]^(.4){P}_(.4){}="P"\ar[dd]_{D\op}&\\
           & {\Cat{InfSL}}\\
{\ER{P}\op}\ar[ru]_(.4){\Sub_{\ER{P}}}^(.4){}="R"
&\ar"P";"R"_{\id{P}}^{\kern-.4ex\cdot}}
$$
in \CM induces an essential equivalence of categories 
$$
-\circ(D,\id{P}):\CM(\Sub_{\ER{P}},\S{X})\equiv\CM(P,\S{X})
$$
for every \cat{X} in \XC.
\end{theorem}

\begin{cor}\label{maincor}
The action of the left biadjoint to the 2-functor
$\XC\longrightarrow\EE$ that takes an exact category to the elementary
existential doctrine of its subobjects is given by \TP{P} on each
elementary existential doctrine $P$.
\end{cor}

\begin{prop}\label{fulc}
If the elementary existential doctrine 
$P:\cat{C}\op\longrightarrow\Cat{InfSL}$
has full comprehensions, then the inclusion of \ER{P} into \TP{P} is
an equivalence of categories.
\end{prop}

\begin{proof}
It is sufficient to note that, since $P$ has full comprehensions, for
any $\alpha$ in $P(A)$, one has $\alpha=\D_{\cmp{\alpha}}\tt$. Hence
$$\EE(P,\S{X})\equiv\CM(P,\S{X})$$
for any regular category \cat{X}.\end{proof}

\begin{remark}
The statement in \ref{fulc} holds also when the elementary
existential doctrine $P$ has just weak full comprehension.
%% There is a similar result for $P$ with only \emph{weak} full
%% comprehensions.
We suspect that this is related to the
analysis carried out by Jonas Frey on pre-equipments of triposes in
\cite{FreyJ:2catatt}.
\end{remark}

\section{Comparing quotient completions}

In \cite{MaiettiME:eleqc}, the authors considered a completion for
quotients of an elementary doctrine
$P:\cat{C}\op\longrightarrow\Cat{InfSL}$ which compares with the one 
presented in the previous section when $P$ is also existential.

Recall from \loccit that the elementary quotient completion
$\cat{Q}_P$ of $P$ consists of 
\begin{description}
\item[objects] which are pairs $(A,\rho)$ such that $\rho$
is a $P$-equivalence relation on $A$,
\item[an arrow {$\ec{f}:(A,\rho)\to(B,\sigma)$}] is an equivalence
class of arrows $f:A\to B$ in \cat{C} such that 
$\rho\leq P_{f\times f}(\sigma)$ in $P(A\times A)$ with
respect to the relation determined by the condition that 
$\rho\leq P_{f\times g}(\sigma)$
\end{description}
Composition is given by that of \cat{C} on representatives, and
identities are represented by identities of \cat{C}.

The indexed partial inf-semilattice
$\Q{P}:\cat{Q}_P\op\longrightarrow\Cat{InfSL}$
on $\cat{Q}_P$ is defined
on an object $(A,\rho)$  as
$$
\Q{P}(A,\rho)\colon=\des{\rho}
$$
where $\des{\rho}$ is the sub-order of
$P(A)$ on those $\alpha$ such that
$P_{\pr_1}(\alpha)\Land\rho\leq P_{\pr_2}(\alpha)$,
where $\pr_1,\pr_2:A\times A\to A$ are the projections.

By Theorem 6.1 in \cite{MaiettiME:quofcm}, when $P$ is existential
with (weak) full comprehensions,
also \Q{P} is existential. Since clearly $\Sub_{\ER{P}}$ has
quotients, there is a canonical arrow
$$
\xymatrix@C=4em@R=1em{
{\cat{Q}_P\op}\ar[rd]^(.4){\Q{P}}_(.4){}="P"\ar[dd]_{L\op}&\\
           & {\Cat{InfSL}}\\
{\ER{P}\op}\ar[ru]_(.4){\Sub_{\TP{P}}}^(.4){}="R"
&\ar"P";"R"_{\ell}^{\kern-.4ex\cdot}}
$$
of elementary existential doctrines which preserves quotients.

It is easy to see that the action of $L$ on objects is the identity
and that the components of $\ell$ are identity homomorphisms. And for
an arrow $\ec{f}:(A,\rho)\to(B,\sigma)$ in $\cat{Q}_P$
$$L\ec{f}
=\D_{\pr_2}(\D_{(\pr_1,f\circ\pr_2)}(\rho)\Land
 P_{(\pr_2,\pr_3)}(\sigma))
=\D_{\pr'_2}(P_{(\pr'_1,\pr'_2)}(\rho)\Land
 P_{(f\circ\pr'_2,\pr'_3)}(\sigma))$$
where $\pr$ denotes a projection from $A\times A\times B$ and
$\pr'$ denotes a projection from $A\times B\times B$.
Note that the construction of $L$ can be performed for any elementary
existential doctrine $P$ and that clearly $L$ is faithful.

\begin{exm}
An interesting example of the comparison above appears in
\cite{FreyJ:2catatt} applied to 
the doctrine $\P{P}:\Gr(P)\op\longrightarrow\Cat{InfSL}$ for 
$P:\cat{C}\op\longrightarrow\Cat{InfSL}$ a tripos and it is used to
analyze the tripos-to-topos construction in a refined 2-categorical
setup of pre-equipments.
\end{exm}

%% \begin{theorem}
%% Consider an elementary existential doctrine
%% $P:\cat{C}\op\longrightarrow\Cat{InfSL}$
%% with full comprehensions. 
%% If, for every object $A$ and $B$ and for every
%% $\alpha$ in $P(A\times B)$ such that $\tt_A\leq\D_{\pr}(\alpha)$ where
%% $\pr:A\times B\to A$ is the projection, there is an arrow $w:A\to B$
%% such that $\tt_A\leq P_{(\id{A},w)}(\alpha)$, then the functor
%% $L:\cat{Q}_P\to\ER{P}$ is an equivalence.
%% \end{theorem}

\begin{theorem}\label{axc}
Suppose that $P:\cat{C}\op\longrightarrow\Cat{InfSL}$
is an elementary existential doctrine with weak full comprehensions. 
Suppose moreover that, for every object $A$ and $B$ and for every
$\alpha$ in $P(A\times B)$ such that $\tt_A\leq\D_{\pr}(\alpha)$ where
$\pr:A\times B\to A$ is the first projection, there is an arrow
$w:A\to B$ in \cat{C} such that $\tt_A\leq
P_{(\id{A},w)}(\alpha)$. Then the functor $L:\cat{Q}_P\to\ER{P}$ is an
equivalence.
\end{theorem}

\begin{proof}
There is only to prove that $L$ is full. So, given an arrow
$\phi:\ple{A,\rho}\to\ple{B,\sigma}$ in \ER{P}, it is
$\tt_A\leq\D_{p_2}(\phi)$. By hypothesis, there is $f:A\to B$ in
\cat{C} such that $\tt_A\leq P_{(\id{A},f)}(\phi)$, or equivalently
$\D_{\id{A}\times f}(\delta_A)\leq\phi$.
It is then easy to see that $\phi=L\ec{f}$.\end{proof}

\begin{remark}
For $P:\cat{C}\op\longrightarrow\Cat{InfSL}$
an elementary existential doctrine with full comprehensions,
it is possible to prove a converse to \ref{axc} under the further
hypothesis that every reflexive $P$-relation has a smallest transitive
extension, \ie for every object $C$ in \cat{C} and 
every object $\zeta$ in $P(C\times C)$ such that
$\delta_C\leq\zeta$, there is an object
$\tr\zeta$ in $P(C\times C)$ such that 
$$\zeta\leq\tr\zeta\qquad
\fp{<p_1,p_2>}(\tr\zeta)\Land\fp{<p_2,p_3>}(\tr\zeta)
\leq\fp{<p_1,p_3>}(\tr\zeta)$$
where $p_1,p_2,p_3:C\times C\to C$ are the
projections, and $\tr\zeta$ is smallest with those three properties.

It is easy to see that $\tr\zeta$ is symmetric when $\zeta$ is such.

Given $\alpha$ in $P(A\times B)$ such that
$\tt_A\leq\D_{\pr_1}(\alpha)$, we may assume with no loss of
generality that $\tt_B\leq\D_{\pr_2}(\alpha)$ since $P$ has full
comprehensions---$\pr_1:A\times B\to A$ and $\pr_2:A\times B\to B$ are 
the two projections. The $P$-relation
$\zeta\colon=\D_{\pr'_1}(P_{(\pr'_1,\pr'_2)}(\alpha)\Land
P_{(\pr'_1,\pr'_3)}(\alpha))$ is reflexive and symmetric in 
$P(B\times B)$. Hence $\alpha:(A,\delta_A)\to(B,\tr\zeta)$ is an arrow
in \ER{P}. Since $L$ is an equivalence, there is
$\ec{w}:(A,\delta_A)\to(B,\tr\zeta)$ in $\cat{Q}_P$ such that
$L\ec{w}=\alpha$, thus $\tt_A\leq P_{(\id{A},w)}(\alpha)$.
\end{remark}

\begin{exms}
The leading example of exact completion satisfying the hypothesis of
\ref{axc}  is that of  exact completion of a category with
products and weak pullback
\cite{CarboniA:freecl,CarboniA:somfcr,CarboniA:regec}. It is 
$\ER{\Psi}$ where $\Psi:\cat{S}\op\longrightarrow\Cat{InfSL}$ is the
functor of weak subobjects.

Another relevant doctrine for \ref{axc} is $\F$ in
 \cite{MaiettiME:quofcm}
giving rise to the total
setoid model of Martin-L{\"o}f's type theory in \cite{PMTT}.

Note also that the second stage of the construction of Joyal's arithmetic
universes in \cite{MaiettiM:joyaul}, which is the category of
decidable predicates 
$\mathrm{Pred}(\cal S)$ on a Skolem theory $\cal S$,
is a regular category and coincides with the base category of the
doctrine obtained by adding full comprehension and forcing
extensionality in the sense of \cite{MaiettiME:eleqc} to the
elementary doctrine of decidable predicates on the Skolem category
$\cal S$. Since epis split in $\mathrm{Pred}(\cal S)$, this is an
example where the hypothesis of \ref{axc} holds for the doctrine of
subobjects of the regular category $\mathrm{Pred}(\cal S)$.
\end{exms}

\raggedright\let\tt\ttori
\bibliographystyle{chicago}
\bibliography{biblio,RosoliniG,procs}
\end{document}